\documentclass[a4paper,11pt]{amsart}
\usepackage{amsmath}
\usepackage{amsthm}
\usepackage{amsfonts}
\usepackage{amssymb}
\usepackage{latexsym}
\usepackage{mathrsfs}

\usepackage[usenames]{color}
\usepackage{graphicx}
\usepackage{epsfig}
\usepackage{psfrag}
\newtheorem{theorem}{Theorem}[section]{\bf}{\it}
\newtheorem{lemma}[theorem]{Lemma}{\bf}{\it}
\newtheorem{proposition}[theorem]{Proposition}{\bf}{\it}
\newtheorem{corollary}[theorem]{Corollary}{\bf}{\it}
{\bf}{\it} % Main Theorems, numbered A,B,...
{\bf}{\it}
{\bf}{\it}

\newtheorem*{theorem*}{Theorem}

\newtheorem*{namedtheorem}{\theoremname}
\newcommand{\theoremname}{testing}

{\bf}{\it}
{\bf}{\it}

\theoremstyle{remark}
\theoremstyle{definition}
 
\theoremstyle{remark}

\numberwithin{equation}{section}

\newcommand{\R}{\mathbb R}
\newcommand{\Z}{\mathbb Z}

\newcommand{\loc}{{\operatorname{loc}}}

\renewcommand{\div}{\operatorname{div}}

\newcommand{\pr}{{\operatorname{pr}}}
\newdimen\vintkern\vintkern11pt
\def\vint{-\kern-\vintkern\int}

\newcommand{\dx}{\;\mathrm{d}x}
\newcommand{\dy}{\;\mathrm{d}y}

\newcommand{\norm}[1]{\lVert #1 \rVert}

\newcommand{\abs}[1]{\lvert#1\rvert}

\newcommand{\grad}{\nabla}

\newcommand{\bS}{\mathbb{S}}

\newcommand{\vol}{\mathrm{vol}}

\DeclareMathOperator{\Div}{div}

\begin{document}

\title[]{Quasiregular curves: H\"older continuity and higher integrability}
%\date{\today}

\author{Jani Onninen}
\address{Department of Mathematics, Syracuse University, Syracuse,
NY 13244, USA \and  Department of Mathematics and Statistics, P.O.Box 35 (MaD) FI-40014 University of Jyv\"askyl\"a, Finland}
\email{jkonnine@syr.edu}

\author{Pekka Pankka}
\address{Department of Mathematics and Statistics, P.O. Box 68 (Pietari Kalmin katu 5), FI-00014 University of Helsinki, Finland}
\email{pekka.pankka@helsinki.fi}

\thanks{This work was supported in part by the Academy of Finland project \#297258 and the NSF grant  DMS-1700274.}

\date{\today}

\dedicatory{Dedicated to Pekka Koskela on the occasion of his 60th birthday.}

\begin{abstract}
We show that a $K$-quasiregular $\omega$-curve from a Euclidean domain to a Euclidean space with respect to a covector $\omega$ is locally $(1/K)(\norm{\omega}/|\omega|_{\ell_1})$-H\"older continuous. We also show that quasiregular curves enjoy higher integrability.
\end{abstract}

\maketitle

\setcounter{tocdepth}{3}
%\tableofcontents

\section{Introduction}

The first breakthrough in the theory of quasiregular mappings (or mappings of bounded distortion) is  Reshetnyak's theorem on sharp H\"older continuity: \emph{Let $\Omega \subset \R^n$ be a domain. A $K$-quasiregular mapping $f\colon \Omega \to \R^n$ for $K\ge 1$ is locally $1/K$-H\"older continuous}, see Reshetnyak \cite{Reshetnyak} and also \cite[Corollary II.1]{Reshetnyak-book}. Such H\"older continuity properties of quasiconformal mappings in the plane were
first established by Morrey~\cite{Morrey}. 

Recall that a mapping $f\colon M \to N$ between oriented Riemannian $n$-manifolds is \emph{$K$-quasiregular} if $f$ belongs to the Sobolev space $W^{1,n}_\loc(M, N)$ and satisfies the distortion inequality
\begin{equation}\label{eq:distortion}
\norm{Df}^n \le K J_f 
\end{equation}
almost everywhere in $M$, where $\norm{Df}$ is the operator norm and $J_f$ the Jacobian determinant of $f$.

In the last 20 years the studies of mappings of finite distortion have emerged into  the Geometric Function Theory (GFT)~\cite{Astala-Iwaniec-Martin-book, Hencl-Koskela-book, Iwaniec-Martin-book}. This theory arose from the need to extend the ideas and applications
of the classical theory of quasiregular mappings to the degenerate elliptic setting where the constant $K$ in~\eqref{eq:distortion} is replaced by a finite  function $K \colon M \to [0,\infty)$. There one
finds concrete applications in materials science, particularly nonlinear elasticity and critical phase
phenomena, and in the calculus of variations.  Some bounds on the distortion function $K$ are needed to obtain a viable theory. In the Euclidean degenerated setting, the continuity properties of mappings of finite distortion  under distortion bounds of exponential
type were obtained in~\cite{Iwaniec-Koskela-Onninen-Invent}. The sharp modulus of continuity estimates for such mappings were given in~\cite{Onninen-Zhong-Continuity}, see also~\cite{Koskela-Onninen-continuity}. The paper~\cite{Iwaniec-Koskela-Onninen-Invent} in addition to starting a systematic studies of mappings of finite distortion  in GFT it also started  the naming scheme of the series of papers, see e.g.~\cite{Akkinen-Guo-mfd, Campbell-Hencl-mfd, Clop-Herron-mfd, Faraco-Koskela-Zhong-mfd,  Guo-mfd0, Guo-mfd, Hencl-Maly-mfd, Herron-Koskela-mfd, Holopainen-Pankka-mfd, IKMS-mfd, IKO-mfd,  Kallunki-mfd,  KKM-mfd-1, KKM-mfd, KKMOZ-mfd,  Kirsila-mfd, Klerptl-mfd, Koskela-Maly-mfd, Koskela-Onninen-continuity, Koskela-Onninen-mfd, Koskela-Onninen-Rajala-mfd-1,  Koskela-Onninen-Rajala-mfd, Koskela-Rajala-mfd,   Koskela-Takkinen-mfd, Koskela-Zapadinskaya-Zurcher-mfd, Onninen-mfd-1, Onninen-mfd, Onninen-Zhong,  Rajala-mfd-1, Rajala-mfd-2, Rajala-mfd}. This paper follows such a scheme.
 
\bigskip

%In this note, we prove a result, analogous to Reshetnyak's sharp H\"older continuity theorem, for quasiregular curves. 

In this note we prove H\"older continuity and higher integrability of quasiregular curves. A mapping $f\colon M\to N$ between Riemannian manifolds is a \emph{$K$-quasiregular $\omega$-curve for $K\ge 1$ and an $n$-volume form $\omega \in \Omega^n(N)$} if $M$ is oriented, $n=\dim M \le \dim N$, $f$ belongs to the Sobolev space $W^{1,n}_\loc(M,N)$ and 
\[
(\norm{\omega}\circ f)\norm{Df}^n \le K \star (f^*\omega)
\]
almost everywhere in $M$, where $\norm{\omega}\colon N \to [0,\infty)$ is the pointwise comass norm of the form $\omega$ and $\star$ is the Hodge star operator on $M$. Here, a form $\omega\in \Omega^n(N)$ is an \emph{$n$-volume form} if $\omega$ is closed and non-vanishing, that is, $d\omega =0$ and $\omega_y \ne 0$ for each $y\in N$. 

We refer to \cite{Pankka2019} for discussion on the definition of quasiregular curves. We merely note here that quasiregular mappings are quasiregular curves and that holomorphic curves are $1$-quasiregular curves.

%Our main theorem is that quasiregular curves have the same H\"older regularity as the quasiregular mappings.

Our main theorem is the H\"older regularity of a quasiregular $\omega$-curve in the case of the constant coefficient form $\omega$. Note that, in the following statement, we identify $n$-covectors in $\bigwedge^n \R^m$ with constant coeffient $n$-volume forms in $\R^m$.

\begin{theorem}
\label{thm:main-vague}
Let $\Omega \subset \R^n$ be a domain, $K \ge 1$, and let $\omega\in \bigwedge^n \R^m$ be an $n$-volume form. Then a $K$-quasiregular $\omega$-curve $f\colon \Omega \to \R^m$ is locally $\alpha$-H\"older continuous for $\alpha=\alpha(K,\omega) = (1/K)(\norm{\omega}/|\omega|_{\ell_1})$. 
\end{theorem}
Here $|\omega|_{\ell_1}$ is the $\ell_1$-norm of the covector $\omega$; see Section \ref{sec:Notation}. For simple covectors, we recover the exponent $1/K$, which follows also from the local characterization of quasiregular curves with respect to simple covectors, see \cite{Pankka2019}. We expect that the H\"older exponent $\alpha(K,\omega)$ is not sharp in general. In fact, all examples of quasiregular curves we know are $1/K$-H\"older continuous.

Since a quasiregular curve is locally a quasiregular curve with respect to a constant coeffcient form by \cite[Lemma 5.2]{Pankka2019}, we obtain that quasiregular curves between Riemannian manifolds are locally H\"older continuous. We record this observation as a corollary.

\begin{corollary}
\label{cor:Holder}
Let $M$ and $N$ be Riemannian $n$ and $m$-manifolds, respectively, for $n\le m$, and let $\omega\in \Omega^n(N)$ be an $n$-volume form. Then each $K$-quasiregular $\omega$-curve $M\to N$ is locally $\alpha(K',\omega)$-H\"older continuous for each $K'>K$. 
\end{corollary}
\begin{proof}
Let $K'' \in (K,K')$ and let $\varepsilon>0$ be a constant for which $(1+\varepsilon)^{4n} < K''/K$. Let $x\in M$ and let $\varphi \colon U \to \R^n$ and $\psi \colon V\to \R^m$ be smooth $(1+\varepsilon)$-charts of $M$ and $N$ at $x$ and $f(x)$, respectively, having the property that $fU \subset V$. Then $h = \psi \circ f \circ \varphi^{-1} \colon \varphi U \to \R^m$ is a $K''$-quasiregular $\widetilde \omega$-curve for $\widetilde \omega = (\psi^{-1})^* \omega$. By \cite[Lemma 5.2]{Pankka2019}, for each $x\in M$, then $h$ is a $K'$-quasiregular $\widetilde \omega_x$-curve with respect to the covector $\widetilde \omega_x$ in a neighborhood of $x$. The claim follows now from Theorem \ref{thm:main-vague}.
\end{proof}

In the proof of Theorem \ref{thm:main-vague} we mimic the lines of reasoning of the original proofs of Reshetnyak's theorem by Morrey~\cite{Morrey} and Reshetnyak~\cite{Reshetnyak}. For quasiregular $\omega$-curves $\Omega\to \R^m$, where $\omega$ is a constant coefficient form or a covector $\omega\in \bigwedge^n \R^m$, we prove a decay estimate on the integrals of  $\star f^*\omega$  of the quasiregular curve $f$  over balls by establishing a differential inequality for these integrals. This is done by employing a suitable isoperimetric inequality.  For this reason, we recall the classical isoperimetric inequality for Sobolev mappings in Section \ref{sec:classical} and derive an $\omega$-isoperimetric inequality in Section \ref{sec:omega}.

\subsection*{Higher integrability of quasiregular curves}
We switch now gears and consider another classical property of quasiregular mappings. Quasiconformal and quasiregular mappings $f\colon \Omega \to \R^n$, $\Omega \subset \R^n$, belong to a higher Sobolev class $W_{\loc}^{1,p} (\Omega)$, $p>n$,  than initially assumed. The sharp exponent $p=p(n,K)$ is not known. A well-known conjecture asserts that
\[
p(n,K)= \frac{nK^\frac{1}{n-1}}{K^\frac{1}{n-1}-1}.
\]
This value, if correct, would be sharp as confirmed  by the radial stretch mapping $f(x)= |x|^\frac{1}{K} \frac{x}{|x|}$. In a seminal work, Astala~\cite{Astala-ACTA} established the sharp exponent  in the planar case $n=2$. There are more recent accounts on the higher intergability results when $n\ge 3 $, we refer here to the celebrated paper of Gehring \cite{Gehring} for the quasiconformal case. In the quasiregular case, we find that the discussion in Bojarski--Iwaniec \cite{Bojarski-Iwaniec} has stood the test of time.

As Bojarski and Iwaniec write in \cite[p.272]{Bojarski-Iwaniec}, the higher integrability of a $K$-quasiregular map $f\colon M \to N$ stems from the double inequality
\[
J_f \le \norm{Df}^n \le K J_f \text{ a.e. in } M
\]
and (standard) harmonic analysis. For a $K$-quasiregular $\omega$-curve $f\colon M\to N$ between Riemannian manifold the analogous double inequality is
\[
\star f^*\omega \le (\norm{\omega}\circ f) \norm{Df}^n \le K \left( \star f^*\omega\right) \text{ a.e. in } M.
\]
The proof of the higher integrability of quasiregular mappings adapts almost synthetically for quasiregular curves.

\begin{theorem}
\label{thm:higher}
Let $f\colon \Omega \to \R^m$ be a $K$-quasiregular $\omega$-curve, where $\omega \in \bigwedge^n \R^m$ is an $n$-volume form. Then $f\in W^{1,p}_\loc(\Omega, \R^n)$ for some $p=p(n,K,\omega)>n$.
\end{theorem}

%Dispite the fact that results of this type are staple theorems in the quasiworld, a concrete interest to this theorem stems from the fact that it yields the almost everywhere differentiability of quasiregular curves. 

As an application, we obtain the almost everywhere differentiability of quasiregular curves. 
The proof of the following corollary from Theorem \ref{thm:higher} is analogous to the proof of Corollary \ref{cor:Holder} from Theorem \ref{thm:main-vague} and we omit the details.

\begin{corollary}
A quasiregular curve between Riemannian manifolds is almost everywhere differentiable.
\end{corollary}

\section{Notation}
\label{sec:Notation}

In what follows, we denote $(e_1,\ldots,e_m)$ the standard orthonormal basis of $\R^m$ and $(e^1,\ldots, e^m)$ its dual basis in $(\R^n)^*$. The $n$th exterior power of $(\R^m)^*$, we denote $\bigwedge^n \R^m$. 

For each multi-index $I=(i_1,\ldots, i_n)$, where $1\le i_1 < \cdots < i_n \le n$, we denote $e^I = e^{i_1}\wedge \cdots \wedge e^{i_n}$. For $n=m$, we also denote 
\[
\vol_{\R^n} = e^1 \wedge \cdots \wedge e^n.
\] 
Note that, for $n$-covectors in $\R^n$, the Hodge star $\star \colon \bigwedge^n \R^n \to \R$, defined by 
\[
(\star \xi) \vol_{\R^n} = \xi
\]
for each $\xi\in \bigwedge^n \R^n$, gives the identification $\bigwedge^n \R^n \cong \bigwedge^0 \R^n = \R$.

In what follows, we also use the Hodge star $\star \colon \bigwedge^{n-1}\R^n \to (\R^n)^*$ to identify $\bigwedge^{n-1}\R^n$ and $\R^n$. This identification of spaces yields an identification of the adjoint $L^\sharp \colon \R^n \to \R^n$ of a linear map $L \colon \R^n\to \R^n$ with the induced map $\bigwedge^{n-1}L \colon \bigwedge^{n-1} \R^n \to \bigwedge^{n-1}\R^n$.

\subsection*{Norms on forms}

In what follows we use the following notations for inner products and norms of covectors and linear maps. For the exterior power $\bigwedge^n \R^m$, we set $\langle \cdot, \cdot \rangle$ to be the natural inner product induced by the standard Euclidean inner product in $\R^n$, that is, $\langle e^I, e^J \rangle = \delta_{IJ}$ for multi-indices $I$ and $J$. We denote $|\cdot|$ the Euclidean norm induced by this inner product.

We also set an $\ell_1$-norm $|\cdot|_{\ell_1}$ in $\bigwedge^n \R^n$ as follows. For $\omega = \sum_I u_I e^I\in \bigwedge^n \R^m$, we set
\[
|\omega|_{\ell_1} = \sum_I |u_I|.
\]

Given a linear map $L \colon V\to W$ between inner product spaces, the operator norm $\norm{L}$ of $L$ is 
\[
\norm{L} = \sup\{ |L(v)| \colon v\in V,\ |v|=1\}.
\]

Finally, for each multi-index $I=(i_1,\ldots, i_n)$, let $\pi_I \colon \R^m \to \R^n$ be the corresponding projection $(x_1,\ldots, x_m) \mapsto (x_{i_1},\ldots, x_{i_n})$. Then $\omega = \sum_I u_I e^I$ is the covector 
\[
\omega = \sum_I u_I \pi_I^*(\vol_{\R^n}).
\]

\section{Classical isoperimetric inequality for Sobolev maps}
\label{sec:classical}

In this section we recall and prove the classical isoperimetric inequality for Sobolev mappings; see, for example, Reshetnyak \cite[Lemma II.1.2.]{Reshetnyak-book} for a more detailed account.

\begin{theorem}
\label{thm:isoperimetric}
Let $\Omega \subset \R^n$ be a domain and $B_R=B^n(x_\circ,R) \subset \Omega$ a ball. Let also $f \colon \Omega \to \R^n$ be a Sobolev map in  $W_{\loc}^{1,n} (\Omega, \R^n)$. Then, for almost every $r\in (0,R)$, we have that
\begin{equation}
\label{eq:isoperimetric}
\left|  \int_{B_r} J_f \right| \le (n \sqrt[n-1]{\omega_{n-1}})^{-1} 
\left(\int_{\partial B_r} \norm{D^\sharp f} \right)^\frac{n}{n-1},
\end{equation}
where $\omega_{n-1}$ is the $(n-1)$-dimensional area of the unit sphere $\bS^{n-1}$ in $\R^n$.
\end{theorem}

This integral form of the isoperimetric inequality stems from the familiar geometric form of the isoperimetric inequality
\begin{equation}\label{eq:classicaliso}
n^{n-1} \omega_{n-1} \abs{U}^{n-1} \le \abs{\partial U}^n,
\end{equation}
where $\abs{U}$ stands for the volume of a domain $U\subset \R^n$ and $\abs{\partial U}$ is its $(n-1)$-dimensional surface area. The constant $\omega_{n-1}$ is the $(n-1)$-dimensional surface area of the unit sphere $\bS^{n-1}=\partial B^n(0,1)$.

To motivate the integral form of the inequality, we consider first the case of diffeomorphisms. Let $f \colon B_r \to U$ be a diffeomorphism of a ball $B_r=B^n(x_\circ , r) \subset \R^n$ onto $U \subset \R^n$, then 
\[
\abs{U}=\left| \int_{B_r} J_f(x) \dx  \right|
\]
and
\[
\abs{\partial U} \le \int_{\partial B_r} \norm{D^\sharp f(x)} \dx;
\]
here $D^\sharp f(x)$ stands for the cofactor matrix of the differential matrix $Df(x)$; recall that identification $\bigwedge^{n-1} \R^n \cong \R^n$ yields the identification $D^\sharp f(x) = \wedge^{n-1}Df(x)$. 

Having these integral representations for the volume and area, we obtain the integral form of the isoperimetric inequality, namely
\begin{equation}\label{eq:isoperimetric}
n \omega_{n-1}
\left|  \int_{B_r} J_f(x) \dx \right|^{n-1} \le \left(\int_{\partial B_r} \norm{D^\sharp f(x)} \dx \right)^n. 
\end{equation}

The same isoperimetric inequality holds for all mappings in $W^{1,n}_\loc(\Omega, \R^n)$. The proof is based on three tools: integration by parts, local degree, and functions of bounded variation.

\subsection{Integration by parts}

Let $f\colon \Omega \to \R^n$ be a mapping in $W^{1,n}_\loc(\Omega, \R^n)$. Then the Jacobian $J_f$ of $f$ obeys the rule of integration by parts, that is,
\begin{equation}\label{eq:integrationbyparts}
\int_\Omega \varphi J_f = \int_\Omega \varphi df_1\wedge \cdots \wedge df_n 
= - \int_\Omega f_i df_1 \wedge \cdots \wedge df_{n-1} \wedge d\varphi \wedge df_{i+1}\wedge \cdots \wedge df_n
\end{equation}
is valid for every test function $\varphi \in C^\infty_0 (\Omega)$ and each index $i=1, \dots , n$. 

For the surface area term, the integration by parts takes the following form, which we record as a lemma.

\begin{lemma}\label{lem:dividentity}
Let $f\colon \Omega\to \R^n$ be a mapping in the Sobolev class $W^{1,n}_\loc(\Omega, \R^n)$ and $u\in C_0^1 (\R^n, \R^n)$.  Then 
\begin{equation}\label{eq:div}
\Div \big((u\circ f) D^\sharp f \big) = \big((\Div u) \circ  f\big) J_f
\end{equation}
in the sense of distributions. 
\end{lemma}

\begin{proof}
Suppose first that $u \colon \R^n \to \R^n$ is the map
\[
y \mapsto (0, \dots ,0 , u_i(y), 0, \dots , 0 ), 
\]
where $u_i \in C_0^1(\R^n)$ and $i \in \{1, \dots , n\}$, and define 
\[
F=(f_1, \dots, f_{i-1}, u\circ f , f_{i+1}, \dots, f_n) \colon \R^n \to \R^n.
\] 
Let also $\varphi \in C^\infty_0 (\Omega)$. Then~\eqref{eq:integrationbyparts} gives
\[ 
\begin{split} 
\int_\Omega \varphi J_F 
& = - \int_\Omega F_i  df_1 \wedge \cdots \wedge df_{n-1}\wedge d\varphi \wedge df_{i+1} \wedge \cdots \wedge df_n\\
&  = - \int_\Omega \langle  u(f(x))  D^\sharp f(x), \nabla \varphi (x) \rangle \dx. 
\end{split} 
\]
Since
\[
\int_\Omega \varphi (x) J_F(x) \dx = \int_\Omega (\Div u) (f(x)) J_f(x) \varphi (x)\dx 
\]
we have that \eqref{eq:div} follows for $u=(0, \dots ,0 , u_i, 0, \dots , 0 )$.
 The general case follows by the coordinate decomposition of $u$.
\end{proof}
In particularly, if $B_R=B^n(x_\circ , R) \subset \Omega$, then Lemma~\ref{lem:dividentity} gives that
\begin{equation}\label{eq:divestimate}
\left| \int_{B_r} \big( \div u \big)\big(  f(x)\big) J_f(x)  \dx  \right| \le \norm{u}_\infty  \int_{\partial B_r}  \abs{D^\sharp f} \quad \textnormal{for a.e. }  r\in (0,R).
\end{equation}
Indeed, choose a mollifier $\Phi \in C^\infty_0 (B(0,1))$ and let $\Phi_j(x) = j^n \Phi (jx) $ and $\varphi_j$ a convolution approximation of the  characteristic function $\chi_{B(x_\circ , r-1/j)}$; that is, $\varphi_j = \Phi_j \ast  \chi_{B(x_\circ , r-1/j)}$, see \cite[Formula (4.6)]{Iwaniec-Martin-book}. Then $\varphi_j \in C^\infty_0 (\Omega)$ when $j$ is sufficiently large and $\sup \{|\nabla \varphi_j (x)|  \colon x\in \Omega\} \le j$. According to~\eqref{eq:div} we have
\[ 
\begin{split}
  \left|\int_{B_r}   (\Div u) (f(x)) J_f (x) \varphi_j (x) \right | & = \left| -  \int^r_{r-\frac{1}{j}} \int_{\partial B_r}    \langle  u(f(x))  D^\sharp f(x), \nabla \varphi _j(x) \rangle     \right |\\
  & \le \norm{u}_\infty \,  j   \int^r_{r-\frac{1}{j}} \int_{\partial B_r} \abs{D^\sharp f}. 
  \end{split}  \]
Letting $j \to \infty$ and applying the Lebesgue differentiation theorem, we conclude the asserted estimate~\eqref{eq:divestimate}. 

\subsection{Local degree}

Let $B\subset \R^n$ be a ball and let $g\colon \overline{B} \to \R^n $ be a continuous mapping. For every $y_\circ \in \R^n \setminus g(\partial B)$ the {\it Brouwer degree} $\deg(g, B, y_\circ)$ of $g$ with respect to $B$ at $y_\circ$ is a well-defined integer defined as follows. Let $\Omega\subset \R^n \setminus g(\partial B)$ be the $y_\circ$-component of $\R^n \setminus g(\partial B)$ and let $\tilde B = g^{-1}(\Omega)\cap B$. Let also $\iota \colon \tilde B \hookrightarrow B$ be the natural inclusion and let $c_\Omega$ and $c_B$ the generators of the compactly supported Alexander--Spanier cohomology groups $H^n_c(\Omega;\Z)$ and $H^n_c(B;\Z)$, respectively. We may assume that $c_\Omega$ and $c_B$ are fixed so that the orientations of $\Omega$ and $B$ given by $c_\Omega$ and $c_B$ agree with the orientation defined by an orientation class $c_{\R^n}$ of $\R^n$. Then 
\[
\deg(g, B, y_\circ) c_B = \iota^* (g|_{\tilde B})^* c_\Omega.
\]
The Brouwer degree depends only on the boundary values of $g$ in the sense that, if $\tilde g \colon \overline{B}\to \R^n$ is a continuous map satisfying $\tilde g|_{\partial B} = g|_{\partial B}$, then $\deg(y_\circ, g,B) = \deg(y_\circ, g, B)$. Furthermore, if $g\in C^1 (B, \R^n) \cap C^0 (\overline{B}, \R^n)$ and $V$ is a connected component of $\R^n \setminus g(\partial B_r)$ containing $y_\circ$, then we have
\[
\deg(g, B, y_\circ) = \int_{B} \rho (g(x)) J_g(x) \dx = \int_B g^*(\rho \vol_{\R^n}), 
\]
where $\rho\in C_0(V)$ is a nonnegative continuous function satisfying $\int_V \rho (y) \dy =1$. This last statement follows from the identification of the compactly supported Alexander--Spanier cohomology $H^n_c(\cdot; \R) = H^*_c(\cdot;\Z)\otimes \R$ with the compactly supported de Rham cohomology $H^*_{\text{dR},c}(\cdot)$.

\subsection{Proof of Theorem \ref{thm:isoperimetric}}

By approximating $f$, it is enough to prove~\eqref{eq:isoperimetric} for smooth mappings $f \colon \Omega \to \R^n$. We recall that the classical change of variables formula for a continuous function $v \colon \R^n \to \R$ states that 
\begin{equation}\label{eq:cofv} 
\int_{B} (v\circ f) J_f = \int_{\R^n} v(y) \deg (f,B, y) \dy. 
\end{equation}
Applying the identity~\eqref{eq:cofv} with $v=\Div u$ and combining this with~\eqref{eq:divestimate} we obtain
\[ 
\left|\int_{\R^n} \Div u (y) \deg (f,B_r, y) \dy  \right | \le \norm{u}_\infty \int_{\partial B_r}  \abs{D^\sharp f}  
\]
for an arbitrary $u \in C_0^1 (\R^n, \R^n)$. Hence the function $y \mapsto \deg (f,B_r, y) $ has bounded variation and we have the inequality
\begin{equation}\label{eq:bvdivest}
\left( \int_{\R^n} \abs{\deg (f,B_r, y)}^\frac{n}{n-1} \dy \right)^\frac{n-1}{n} 
\le (n \sqrt[n-1]{\omega_{n-1}})^{\frac{n}{1-n}}  \int_{\partial B_r}  \abs{D^\sharp f(x)} \dx \, . 
\end{equation}
 It is worth nothing that the use of the Sobolev inequality~\eqref{eq:bvdivest} comes as no surprise. Indeed, the Sobolev inequality
\[
n^\frac{n-1}{n} \omega^\frac{1}{n}_{n-1} \norm{g}_{\frac{n}{n-1}} \le \abs{Dg} (\R^n) 
\]
for functions of bounded variation $g\colon \R^n \to \R$ is equivalent with the classical isoperimetric inequality~\eqref{eq:classicaliso}.
%to hold for all  functions $g \colon \R^n \to \R$ of bounded variation equivalent with the classical isoperimetric inequality~\eqref{eq:classicaliso}. 
Here $ \abs{Dg} (\R^n) $ stands for the total variation of the distributional derivative $Dg$ see e.g.~Evans and Gariepy \cite[Section 5.6]{Evans-Gariepy-book}

Since the function $y \mapsto \deg (f,B_r, y)$ is integer valued, we further have that
\[ 
\abs{\deg (f,B_r, y)} \le  \abs{\deg (f,B_r, y)}^\frac{n}{n-1}. 
\]
for each $y\in \R^n \setminus f(\partial B_r)$. Thus
\[ 
\left| \int_{\R^n}  \deg (f,B_r, y) \dy \right|
\le (n \sqrt[n-1]{\omega_{n-1}})^{-1} \left( \int_{\partial B_r} \abs{D^\sharp f}\right)^\frac{n}{n-1}.  
\]
Applying~\eqref{eq:cofv} again, this time with $v\equiv 1$, we obtain the desired inequality
\[ 
\left| \int_{B_r}  J_f \right | \le  (n \sqrt[n-1]{\omega_{n-1}})^{-1} \left( \int_{\partial B_r}  \abs{D^\sharp f} \right)^\frac{n}{n-1}.
\]
This concludes the proof.
%\end{proof}

\section{An $\omega$-isoperimetric inequality for Sobolev maps}
\label{sec:omega}

The proof of the H\"older continuity of the quasiregular curves is based on a variant  of the classical isoperimetric inequality for Sobolev maps adapted to $n$-volume forms. 

\begin{proposition}
\label{prop:omega-isoperimetry}
Let $\Omega \subset \R^n$ be a domain, $B_R=B(x_\circ , R)\subset \Omega$ a ball, and $\omega \in \bigwedge^n \R^m$ an $n$-covector. Then a Sobolev $W^{1,n}_\loc(\Omega,\R^m)$ map $f\colon \Omega \to \R^m$ satisfies 
\begin{equation}
\label{eq:omega-isoperimetry}
\int_{B_r} f^*\omega \le c_n |\omega|_{\ell_1} \left(\int_{\partial B_r} |D^\sharp f| \right)^\frac{n}{n-1} \quad \textnormal{for a.e } r\in (0,R).
\end{equation}
Here $c_n=(n \sqrt[n-1]{\omega_{n-1}})^{-1}$ is the isoperimetric constant. 
\end{proposition}

\begin{proof}
Let $\omega\in \bigwedge^n \R^m$ be the covector
\[
\omega = \sum_I u_I \pr_I^*\vol_{\R^n}.
\]

For each multi-index $I$, let $\lambda_I \colon \R^n \to \R^n$ be the linear map 
\[
(y_1,\ldots, y_n) \mapsto (\varepsilon |u_I|^{1/n} y_1, \ldots, |u_I|^{1/n} y_n), 
\]
where the sign $\varepsilon\in \{\pm 1\}$ is chosen so that $\lambda_I^*\vol_{\R^n} = u_I \vol_{\R^n}$. 

Let also $\pi_I = \lambda_I \circ \pr_I \colon \R^m \to \R^n$ and $f_I = \pi_I \circ f \colon \Omega \to \R^n$. Then
\[
f^*\omega = \sum_I f^*(\lambda_I \circ \pr_I)^* \vol_{\R^n} = \sum_I f^*\pi_I^*\vol_{\R^n} = \sum_I f_I^*\vol_{\R^n}-
\]
Moreover, 
\[
\norm{D^\sharp \pi_I} = \norm{\wedge^{n-1} \pi_I} = |u_I|^{\frac{n-1}{n}}.
\]

By the isoperimetric inequality for Sobolev mappings, we have
\begin{eqnarray*}
\int_{B_r} f^*\omega &=& \int_{B_r} \sum_I f^*(\pi_I^*\vol_{\R^n}) 
= \sum_I \int_{B_r} f_I^*\vol_{\R^n} \\
&=& \sum_I \int_{B_r} J_{f_I}
\le c_n \sum_I \left( \int_{\partial B_r} |D^\sharp f_I| \dx\right)^\frac{n}{n-1}
\end{eqnarray*}
for almost every $r\in (0,R)$, where $c_n>0$ is the isoperimetric constant depending only on $n$.

Since
\[
D^\sharp f_I = \wedge^{n-1} D(\pi_I \circ f) = ((\wedge^{n-1} D\pi_I)\circ f)\cdot(\wedge^{n-1}Df),
\]
we have that 
\begin{eqnarray*}
\left( \int_{\partial B_r} \norm{D^\sharp f_I} \dx\right)^\frac{n}{n-1}
&=& \left( \int_{\partial B_r} (\norm{\wedge^{n-1}D\pi_I}\circ f)\cdot \norm{\wedge^{n-1}Df} \right)^\frac{n}{n-1} \\
&=& \left( \int_{\partial B_r} |u_I|^\frac{n-1}{n} \norm{\wedge^{n-1}Df} \right)^\frac{n}{n-1} \\
&=& |u_I| \left(\int_{\partial B_r} \norm{D^\sharp f} \right)^\frac{n}{n-1}
\end{eqnarray*}
for almost every $r\in (0,R)$. Thus \eqref{eq:omega-isoperimetry} holds.
\end{proof}

\section{Proof of the H\"older continuity}
\label{sec:main-thm-proof}
Let $f\colon \Omega \to \R^m$ be a $K$-quasiregular $\omega$-curve with respect to a covector $\omega \in \bigwedge^n \R^m$ and let $B_R=B(x_\circ , R) \subset \Omega$ be a ball. Morrey's ideas~\cite{Morrey} form the basis for our proof here. A crucial tool in establishing the sharp H\"older exponent is the isoperimetric inequality~\eqref{eq:isoperimetric} which together with distortion inequality, Hadamard's inequality $\norm{D^\sharp f} \le \norm{Df}^{n-1}$, and H\"older's inequality gives
\begin{align*}
\int_{B_r} \norm{\omega} \norm{Df}^n &\le K \int_{B_r} f^*\omega \le  (n \sqrt[n-1]{\omega_{n-1}})^{-1}|\omega|_{\ell_1} K \left(\int_{\partial B_r} \norm{D^\sharp f} \right)^\frac{n}{n-1} \\
&\le   (n \sqrt[n-1]{\omega_{n-1}})^{-1} |\omega|_{\ell_1} K \left(\int_{\partial B_r} \norm{D f}^{n-1} \right)^\frac{n}{n-1}\\
& \le \frac{r}{n}|\omega|_{\ell_1} K \int_{\partial B_r} \norm{D f}^{n} 
\le \frac{r}{n}\frac{|\omega|_{\ell_1}}{\norm{\omega}} K \int_{\partial B_r} \norm{\omega}\norm{Df}^n
\end{align*}
for almost every $r\in (0,R)$. 
Thus
\[
\Phi (r):=\int_{B_r} \norm{Df}^n  \le  \frac{r}{n}\frac{|\omega|_{\ell_1}}{\norm{\omega}}K \int_{\partial B_r} \norm{Df}^n  = \frac{r}{n} \frac{|\omega|_{\ell_1}}{\norm{\omega}} K \,  \Phi' (r)  
\]
and therefore
\[ 
\frac{n}{K} \frac{\norm{\omega}}{|\omega|_{\ell_1}} \frac{d  }{dr} \log r \le \frac{d}{dr} \log \Phi (r).
\]
After integrating this estimate from $r$ to $R$ with respect the variable $r$ we obtain
\[  
\int_{B_r} \norm{Df}^n = \Phi (r) \le \left( \frac{r}{R}\right)^{\frac{n}{K} \frac{\norm{\omega}}{|\omega|_{\ell_1}}}\Phi (R) = \left( \frac{r}{R}\right)^{\frac{n}{K}\frac{\norm{\omega}}{|\omega|_{\ell_1}}} \int_{B_R} \norm{Df}^n.
\]

We record the outcome as a lemma.
\begin{lemma}
Let $f\colon \Omega \to \R^m$ be a $K$-quasiregular mapping and $B(a , 3R) \subset \Omega$ a ball. Then for each ball $B_r=B(x_\circ , r) \subset B(a,2R)$ 
we have 
\begin{equation}\label{eq:morreydecay}
\left(\frac{1}{\abs{B_r}} \int_{B_r} \norm{Df}^n \right)^\frac{1}{n} \le C r^{\frac{1}{K}\frac{\norm{\omega}}{|\omega|_{\ell_1}}-1}, 
\end{equation} 
where the constant $C$ depends on $n, K, R$, and $\int_{B(a,3R)} \norm{Df}^n$.
\end{lemma}
Now it is well known that the hunted local H\"older continuity follows for a Sobolev mapping whose differential lies in the Morrey space~\eqref{eq:morreydecay}.  Our proof is based on the iconic Sobolev met Poincar\'e chain argument~\cite{Hajlasz-Koskela-SobolevmetPoincare}.

\begin{lemma}
Let ${\mathbb  B} \subset \R^n$ be a ball and $g \colon 2\mathbb B \to \R$ a Sobolev function in $W^{1,p} (2 \mathbb B)$ for $1\le p < \infty$. If for every ball $B_r = B(x_\circ , r ) \subset 2\mathbb  B$ we have
 \begin{equation}\label{eq:morreyspaceass}
 \left(\frac{1}{\abs{B_r}} \int_{B_r} \norm{Df(x)}^p \dx\right)^\frac{1}{p} \le C r^{\alpha -1} \qquad 0<\alpha \le 1, 
  \end{equation}
then $g$ is H\"older continuous in $\mathbb B$ with  exponent $\alpha$.
\end{lemma}
\begin{proof}
 Let $x,y \in \mathbb B$ be Lebesgue points of $g$. Write $\mathcal B_i (x)=  B(x, 2^{-i }\abs{x-y})$ for $i\in \{0, 1,  2, \dots\}$ and $g_{\mathcal B_i(x)} = \frac{1}{\mathcal B_i(x)} \int_{\mathcal B_i(x)} g$. Then $g_{\mathcal B_i (x)} \to g(x) $ as $i$ goes to infinity.  The Poincar\'e inequality gives
\[
\begin{split}
\abs{g(x)-g_{\mathcal B_0}(x)} & \le \sum_{i=0}^\infty  \abs{g_{\mathcal B_i}(x) - g_{\mathcal B_{i+1}}(x)   } \\
& \le  \sum_{i=0}^\infty  \frac{1}{\abs{\mathcal B_{i+1} (x)}} \int_{\mathcal B_{i+1} (x)} \abs{g(x)- g_{\mathcal B_i}(x)} \dx \\
& \le C\,  \sum_{i=0}^\infty   \frac{1}{\abs{\mathcal B_{i} (x)}} \int_{\mathcal B_{i} (x)} \abs{g(x)- g_{\mathcal B_i}(x)} \dx \\
& \le C \,  \sum_{i=0}^\infty  2^{-i }\abs{x-y}\left(   \frac{1}{\abs{\mathcal B_{i} (x)}} \int_{\mathcal B_{i} (x)}  \abs{\nabla g (x)}^p \dx \right)^\frac{1}{p}. 
\end{split}
\]
Similarly,
\[ 
\abs{g(y)-g_{\mathcal B_0}(y)}   \le C \,  \sum_{i=0}^\infty  2^{-i }\abs{x-y}\left(   \frac{1}{\abs{\mathcal B_{i} (y)}} \int_{\mathcal B_{i} (y)}  \abs{\nabla g (x)}^p \dx \right)^\frac{1}{p} 
\]
and
\[
\abs{g_{\mathcal B_0}(x) - g_{\mathcal B_0}(y)}  \le C \abs{x-y} \left(   \frac{1}{\abs{2 \mathcal B_{0} (x)}} \int_{2\mathcal B_{0} (x)}  \abs{\nabla g (x)}^p \dx \right)^\frac{1}{p}.   
\]
Combining these with the assumption~\eqref{eq:morreyspaceass} we have
\[
\abs{g(x)-g(y)} \le C \abs{x-y}^\alpha   \left(   \int_{2\mathcal B_{0} (x)}  \abs{\nabla g (x)}^p \dx \right)^\frac{1}{p}   \sum_{i=0}^\infty  (2^{-i})^\alpha. 
\]
The claim follows because the geometric series is convergent. 
\end{proof}

\section{Higher integrability of quasiregular curves}

As in the quasiregular case (see e.g.~\cite{Bojarski-Iwaniec}), the proof of the higher integrability begins with a Caccioppoli inequality. Since we use here another version of the inequality than in \cite{Pankka2019}, we recall here the standard argument.

\begin{lemma}[Caccioppoli's inequality]
\label{lemma:Caccioppoli}
Let $\Omega$ be a domain, $f\colon \Omega \to \R^m$ be a $K$-quasiregular $\omega$-curve, where $\omega \in \Omega^n(\R^m)$ is an $n$-volume form with constant coeffcients. Then, for each cube $B\Subset \Omega$ and for each non-negative function $\varphi \in C^\infty_0(B)$, 
\[
\int_B \varphi^n f^*\omega \le n^n \norm{\omega} K^{n-1}  \int_B |\grad \varphi|^n \left|f(x)-f_B \right|^n \dx,  
\]
where 
\[
f_B = \vint_B f(x) \dx.
\]
\end{lemma}

\begin{proof}
Let $y_\circ = f_B$ for simplicity. Since $\omega$ is closed, it is exact and we may fix an $(n-1)$-form $\tau\in \Omega^{n-1}(\R^m)$ for which $\omega = d\tau$ and $\tau_{y_\circ} = 0$. Then $\tau$ is $\norm{\omega}$-Lipschitz. More precisely, we have that $\norm{\tau}_y \le \norm{\omega} |y-y_\circ|$ for each $y\in \R^m$.

Let $\varphi\in C^\infty_0(B)$ be a non-negative function satisfying $\varphi|_{\frac{1}{2}B} \equiv 1$. Since the function $\star f^*\omega$ is non-negative, we have that 
\begin{eqnarray*}
\int_B \varphi^n f^*\omega 
&=& \int_B \varphi^n f^*d\tau 
= \int_B \varphi^n df^*\tau \\
&=& \int_B d(\varphi^n f^*\tau) - \int_B d\varphi^n \wedge f^*\tau 
= - \int_B d\varphi^n \wedge f^*\tau \\
&\le& \int_B |\grad \varphi^n| (\norm{\tau}\circ f)\norm{Df}^{n-1} \\
&\le& n \norm{\omega} \int_B |\grad \varphi(x)| \left| f(x)-y_\circ \right| \varphi^{n-1} \norm{Df(x)}^{n-1} \dx,
\end{eqnarray*}
where $\norm{\tau}$ is the pointwise comass norm of $\tau$. Thus, by H\"older's inequality, 
\[
\int_B \varphi f^*\omega 
\le n\norm{\omega} \left( \int_B |\grad \varphi|^n \left| f(x)-y_\circ \right|^n \dx \right)^{1/n}  \left( \int_B  \varphi^n \norm{Df}^n \right)^{(n-1)/n}.
\]
Since $f$ is a $K$-quasiregular $\omega$-curve, we have that
\[
\norm{\omega} \left( \int_B  \varphi^n \norm{Df}^n \right)^{(n-1)/n} \le \norm{\omega}^{1/n} \left( \int_B \varphi^n f^*\omega \right)^{(n-1)/n}.
\]
Thus  
\begin{align*}
\left( \int_B \varphi f^*\omega \right)^{1/n}
&\le n \norm{\omega}^{1/n} K^{(n-1)/n} \left( \int_B |\grad \varphi|^n \left| f(x)-y_\circ \right|^n \dx \right)^{1/n}. 
\end{align*}
\end{proof}

The Poincar\'e inequality for Sobolev functions in $W^{1,n}_\loc$ now yields the following corollary.

\begin{lemma}
\label{lemma:Poincare}
Let $\Omega$ be a domain, $f\colon \Omega \to \R^m$ be a $K$-quasiregular $\omega$-curve, where $\omega \in \Omega^n(\R^m)$ is an $n$-volume form with constant coeffcients. Let $B=B^n(x_\circ,r) \subset \Omega$ be a ball. Then there exists a constant $C=C(n,K)>0$ for which
\[
\left(\int_{\frac{1}{2}B} \norm{Df}^n \right)^{1/n} \le \frac{C}{r^{1/n}} \left( \int_B \norm{Df}^{n/2} \right)^{2/n}.
\]
\end{lemma}

\begin{proof}
Let $\varphi \in C^\infty_0(B)$ be the standard cut-off function satisfying $\varphi|_{\frac{1}{2}B} \equiv 1$ and $|\grad \varphi| \le 3/r$. Then by the quasiregularity and the Caccioppoli inequality,
we have the estimate
\begin{align*}
\norm{\omega} \int_{\frac{1}{2}B} \norm{Df}^n 
&\le \int_{\frac{1}{2}B} K f^*\omega 
\le K n^n \norm{\omega} K^{n-1} \frac{3}{r} \int_B |f(x) - f_B|^n \dx.
\end{align*}
Thus, by the Poincar\'e inequality, we have the estimate
\begin{align*}
\left( \int_{\frac{1}{2}B} \norm{Df}^n \right)^{1/n}
&\le \frac{C(n,K)}{r^{1/n}} \left( \int_B |f(x) - f_B|^n \dx\right)^{1/n} \\
&\le  \frac{C(n,K)}{r^{1/n}} \sum_{i=1}^n \left( \int_B |f_i(x) - (f_i)_B|^n \dx\right)^{1/n} \\
&\le  \frac{C(n,K)}{r^{1/n}} \sum_{i=1}^n \left( \int_B \norm{Df_i}^{n/2} \dx\right)^{2/n} \\
&\le \frac{C(n,K)}{r^{1/n}} \left( \int_B \norm{Df}^{n/2} \right)^{2/n};
\end{align*}
here we used the fact that $f-f_B = \left(f_1 - (f_1)_B, \ldots, f_n-(f_n)_B\right)$. 
\end{proof}

The higher integrability of the quasiregular $\omega$-curves with respect to constant coefficient $n$-volume forms now follow with the standard reverse H\"older argument. Before the statement, we recall that, as in the quasiregular case, the in Lemmas \ref{lemma:Caccioppoli} and \ref{lemma:Poincare}, the claims hold for a cube $Q\subset \Omega$ in place of the ball $B$. 

We record the higher integrability of a quasiregular curve --  with respect to a covector -- as follows.

\begin{proposition}
Let $f\colon \Omega \to \R^m$ be a $K$-quasiregular $\omega$-curve for $\omega \in \bigwedge^n \R^m$. Then there exists $p=p(n,K)>n$ and $C=C(n,K,p) \ge 1$ having the property that, for each cube $Q\subset 2Q \subset \Omega$, holds
\[
\left( \int_Q \norm{Df}^p \right)^{1/p} \le C \left( \int_Q \norm{Df}^n \right)^{1/p}.
\]
\end{proposition}
\begin{proof}
Let $Q'=Q'(x,r) \subset Q$ be a subcube. Then, by Lemma \ref{lemma:Poincare}, we have that
\begin{align*}
\left( \vint_{\frac{1}{2}Q'} \norm{Df}^n \right)^{1/n} 
&= \left( \frac{1}{|\frac{1}{2}Q'|}\right)^{1/n} \left( \int_{\frac{1}{2}Q'} \norm{Df}^n \right)^{1/n} \\
&\le \left( \frac{1}{|\frac{1}{2}Q'|}\right)^{1/n} \frac{C(n,K)}{|Q'|^{1/n}} \left(\int_{Q'} \norm{Df}^{n/2} \right)^{2/n}  \\
&= C(n,K) \left( \vint_{Q'} \norm{Df}^{n/2} \right)^{2/n}.
\end{align*}

Let now $u = \norm{Df}^{n/2} \in L^2(Q)$. Then, by Gehring's lemma (see e.g.~\cite[Theorem 4.2]{Bojarski-Iwaniec}), there exists $t>2$ and $C_t > 1$ for which 
\[
\left( \int_{\frac{1}{2} Q'} u^t \right)^{1/t} \le C \left( \int_{Q'} u^2 \right)^{1/2}
\]
for each subcube $Q'\subset Q$. Thus $\norm{Df}\in L^p(Q)$ for $p=t n/2> n$.
\end{proof}

\bibliographystyle{abbrv}
%\bibliography{QRC}

\end{document}